\documentclass[11pt]{amsart}

\usepackage{amscd}
\usepackage{amsmath, amssymb}
\usepackage{amsfonts}
\newcommand{\de}{\partial}
\newcommand{\db}{\overline{\partial}}

\newcommand{\ddbar}{\sqrt{-1} \partial \overline{\partial}}

\newcommand{\ov}[1]{\overline{#1}}
\newcommand{\mn}{\sqrt{-1}}

\newcommand{\ti}[1]{\tilde{#1}}

\newcommand{\vol}{\mathrm{Vol}}
\newcommand{\diam}{\mathrm{diam}}

\renewcommand{\leq}{\leqslant}
\renewcommand{\geq}{\geqslant}

\newtheorem{theorem}{Theorem}[section]

\theoremstyle{definition}

\title{Families of Calabi--Yau manifolds and canonical singularities}
\author{Valentino Tosatti}
\thanks{Research supported in part by a Sloan Research Fellowship and NSF grant DMS-1308988.}
\address{Department of Mathematics, Northwestern University, 2033 Sheridan Road, Evanston, IL 60208}
\email{tosatti@math.northwestern.edu}

\begin{document}
\begin{abstract} Given a polarized family of varieties over $\Delta$, smooth over $\Delta^{\times}$ and with smooth fibers Calabi-Yau, we show that the origin lies at finite Weil--Petersson distance if and only if after a finite base change the family is birational to one with central fiber a Calabi--Yau variety with at worst canonical singularities, answering a question of C.-L. Wang. This condition also implies that the Ricci-flat K\"ahler metrics in the polarization class on the smooth fibers have uniformly bounded diameter, or are uniformly volume non-collapsed.
\end{abstract}
\maketitle

\section{Introduction}
In this note we assume that $\pi:\mathfrak{X}\to\Delta$ is a flat projective family of $n$-dimensional varieties over the unit disc $\Delta\subset\mathbb{C}$, smooth over $\Delta^\times$,
with $\mathfrak{X}$ normal, $K_{\mathfrak{X}}$ Cartier and $K_{\mathfrak{X}/\Delta}\cong \mathcal{O}_{\mathfrak{X}}$. Therefore $X_t=\pi^{-1}(t)$ is a Calabi--Yau manifold for $t\neq 0$.
All our families are assumed to have connected fibers.

Let $\Omega$ be a nowhere vanishing holomorphic section of $K_{\mathfrak{X}/\Delta}$ and let $\Omega_t=\Omega|_{X_t}$ for $t\neq 0$.
Fix also $\mathcal{L}\to\mathfrak{X}$ a holomorphic line bundle with $L_t=\mathcal{L}|_{X_t}$ ample for all $t\in \Delta$. Thanks to the Calabi--Yau theorem, there exists a unique Ricci-flat K\"ahler metric $\omega_t$ on $X_t$ in the class $c_1(L_t)$, for $t\neq 0$.

Recall that in this setup there is a smooth semi-positive definite closed real $(1,1)$ form $\omega_{WP}$ on $\Delta^{\times}$, the Weil--Petersson form, which we can view as a possibly degenerate K\"ahler metric. Explicitly, we have
$$\omega_{WP}=-\mn\de_t\db_t\log\left((-1)^{\frac{n^2}{2}}\int_{X_t}\Omega_t\wedge\ov{\Omega_t}\right).$$
We can then use $\omega_{WP}$ to compute the length of curves in $\Delta^\times$, and so it makes sense to ask whether $0$ lies at finite or infinite distance. Note that in fact to define $\omega_{WP}$ the assumption that $K_{\mathfrak{X}/\Delta}\cong \mathcal{O}_{\mathfrak{X}}$ can be relaxed to just assuming that $K_{X_t}\cong \mathcal{O}_{X_t}$ for all $t\in \Delta^\times$, and taking $\Omega_t$ to be any trivializing section of $K_{X_t}$, since $\omega_{WP}$ is well-defined independent of the choice of $\Omega_t$.

Recall the definition of canonical and log terminal singularities. A normal projective variety $X$ has at worst canonical (resp. log terminal) singularities if the canonical sheaf $K_X$ is $\mathbb{Q}$-Cartier and for any resolution $\pi:\ti{X}\to X$ we have
$$K_{\ti{X}}\sim_{\mathbb{Q}} \pi^*K_X+\sum_E a_E E,$$
with $a_E\geq 0$ (resp. $a_E>-1$), where the sum ranges over all $\pi$-exceptional divisors $E$ (and $a_E\in\mathbb{Q}$).

Similarly, a pair $(X,\Delta)$ of a normal variety $X$ and a divisor $\Delta=\sum_i d_i D_i$ (with $d_i\in\mathbb{Q}$, $0\leq d_i\leq 1$, and the $D_i$'s are distinct prime divisors) has at worst plt (purely log terminal) singularities if $K_X+\Delta$ is $\mathbb{Q}$-Cartier and there exists a log resolution $\pi:\ti{X}\to X$ with
$$K_{\ti{X}}\sim_{\mathbb{Q}} \pi^*(K_X+\Delta)+\sum_E a_E E,$$
where the sum ranges over all prime divisors $E\subset \ti{X}$ (and $a_E\in\mathbb{Q}$), and $a_E>-1$ for all $\pi$-exceptional divisors $E$. It has at worst dlt (divisorially log terminal) singularities if this condition holds only for all $\pi$-exceptional divisors $E$ which lie above the non-snc locus of $(X,\Delta)$ (i.e. where this pair does not have simple normal crossings).

A Calabi--Yau variety with at worst canonical singularities is a normal projective variety $X$ with at worst canonical singularities and with $K_X$ Cartier and trivial.
This can be considered as the natural class of ``singular Calabi--Yau manifolds''. Thanks to \cite{EGZ} they admit a unique singular Ricci-flat K\"ahler metric in each K\"ahler class. Furthermore, we were informed by Y. Odaka that it follows from his work \cite{O1, O2} that a normal projective variety with $K_X$ trivial has at worst canonical singularities if and only if it is K-stable.

Given a family $\mathfrak{X}\to\Delta$ as above, C.-L. Wang proved in \cite{Wa} that if $X_0$ has at worst canonical singularities, then $0$ lies at finite Weil--Petersson distance from $\Delta^\times$. He also asked in \cite[Question 2.12]{Wa} whether a converse is true as well (see also \cite{Wi}), namely whether after a finite base change the family $\mathfrak{X}$ is birational to another family with central fiber a normal Calabi--Yau variety with at worst canonical singularities. In \cite{Wa2} he showed that this holds if the Minimal Model Program can be ran in dimension $n+1$.

Following the terminology of the Minimal Model Program, we will say that a family $\pi:\mathfrak{X}\to\Delta$ as above is relatively minimal if the pair $(\mathfrak{X},X_0)$ is dlt.
Our first result is the following:

\begin{theorem}\label{main} Let $\pi:\mathfrak{X}\to\Delta$ be a flat projective family of $n$-dimensional varieties over the unit disc $\Delta\subset\mathbb{C}$, smooth over
$\Delta^\times$ and
with $K_{\mathfrak{X}/\Delta}\cong \mathcal{O}_{\mathfrak{X}}$. Assume furthermore that this family is relatively minimal. Then, with the above notations, the following are equivalent:
\begin{itemize}
\item[(a)] $X_0$ has at worst canonical singularities
\item[(b)] There exists an irreducible component $D$ of $X_0$ with a resolution of singularities $\ti{D}\to D$
with $H^0(\ti{D},K_{\ti{D}})\neq 0$
\item[(c)] There is a constant $C>0$ such that $(-1)^{\frac{n^2}{2}}\int_{X_t}\Omega_t\wedge\ov{\Omega_t}\leq C$ for all $t\neq 0$
\item[(d)] There is a constant $C>0$ such that $\omega_t^n\geq C^{-1}(-1)^{\frac{n^2}{2}}\Omega_t\wedge\ov{\Omega_t}$ on $X_t$ for all $t\neq 0$.
\end{itemize}
Furthermore, any of $(a)-(d)$ implies the following equivalent statements:
\begin{itemize}
\item[(e)] There is a constant $C>0$ such that $\diam(X_t,\omega_t)\leq C$ for all $t\neq 0$
\item[(f)] There is a constant $C>0$ such that for any $x\in X_t$ and any $0<r\leq \diam(X_t,\omega_t)$ we have $\vol_{\omega_t}B_{\omega_t}(x,r)\geq C^{-1}r^{2n}$ for all
$t\neq 0$.
\end{itemize}
\end{theorem}

Note that by inversion of adjunction \cite[Theorem 5.50]{KM} if a family as above has central fiber $X_0$ with at worst canonical singularities, then up to shrinking the disc $\Delta$ we have that $(\mathfrak{X},X_0)$ is dlt (in fact even plt), so the family is relatively minimal.
Building on the recent advances in the Minimal Model Program \cite{BCHM, Fu, HX, Lai, NX}, we will use Theorem \ref{main} to show that Wang's question \cite[Question 2.12]{Wa} has an affirmative answer in general.

\begin{theorem}\label{cor}
Let $\pi:\mathfrak{X}\to\Delta$ be a flat projective family of $n$-dimensional varieties over the unit disc $\Delta\subset\mathbb{C}$, smooth over
$\Delta^\times$ and
with $K_{X_t}\cong \mathcal{O}_{X_t}$ for all $t\in\Delta^{\times}$. Then $0$ is at finite Weil--Petersson distance from $\Delta^\times$ if and only if
after a finite base change the family is birational to another family with central fiber a Calabi--Yau variety with at worst canonical singularities.
\end{theorem}

In this theorem the family with central fiber a Calabi--Yau variety with at worst canonical singularities is a relatively minimal model of a semistable reduction of
the initial family. Moreover, this family is uniquely determined by its restriction over $\Delta^\times$, as follows for example from \cite[Theorem 2.1]{Al} or \cite[Lemma 7.2]{HX}.

It is an interesting problem to decide whether in Theorem \ref{main} we also have that $(f)$ implies $(a)$\footnote{Added in proof: after this work was completed, Takayama succeeded in proving this implication in \cite{Ta2}, using \cite{DS}, \cite{Ta1} and this work.}.
This would give an algebro-geometric characterization of volume non-collapsing for families of Calabi--Yau manifolds, a problem which was mentioned by Donaldson--Sun \cite[p.2]{DS}.
Note that if $(f)$ holds, then the manifolds $(X_t,\omega_t), t\neq 0$ satisfy the assumptions of \cite[Theorem 1.2]{DS}, and we conclude that they can all be embedded in $\mathbb{CP}^N$ by sections of $L_t^{k}$ for some fixed integers $k,N$. Furthermore, there is a sequence $t_i\to 0$ such that the images
$X_{t_i}\subset \mathbb{CP}^N$ (modulo projective transformations) have as flat limit a normal projective variety $W$ with at worst log-terminal singularities (by \cite[Proposition 4.15]{DS}). If we could prove that $W\cong X_0$, then it would follow that $X_0$ has at worst canonical singularities (since $K_{X_0}$ is Cartier).

If the central fiber $X_0$ of a family as above has at worst canonical singularities, then in fact we know quite a lot about the behavior of the Ricci-flat K\"ahler manifolds $(X_t,\omega_t)$ as $t$ approaches zero.
Thanks to a theorem of Rong--Zhang \cite[Theorem 1.4]{RZ}, $(X_t,\omega_t)$ converge smoothly in the sense of Cheeger--Gromov as $t\to 0$
to an incomplete Ricci-flat K\"ahler metric $\omega_0$ on $X_0^{reg}$. In fact, $\omega_0$ extends to a ``singular Ricci-flat metric'' on $X_0$ in the sense of \cite{EGZ}.
Furthermore, Rong--Zhang (\cite[Theorem 1.1]{RZ} and \cite{RZ2}) show that the metric spaces $(X_t,\omega_t)$ converge in the Gromov--Hausdorff sense to a compact metric space $(Z,d)$ which is isometric to the metric completion of $(X_0^{reg},\omega_0)$. We also have that $Z$ is homeomorphic to $W$ thanks to Donaldson--Sun \cite{DS}.

If the family is relatively minimal and $X_0$ has worse than canonical singularities, e.g. if $X_0$ is a large complex structure limit (in the sense of \cite{Mo}), then $\diam(X_t,\omega_t)$ is expected to diverge\footnote{Added in proof: this has now been proved by Takayama \cite{Ta2}.}. In particular, the rescaled unit-diameter Ricci-flat K\"ahler metrics $\ti{\omega}_t=(\diam(X_t,\omega_t))^{-2}\omega_t$ would be volume collapsing. In the case of a large complex structure limit, Kontsevich--Soibelman \cite{KS} and Gross--Wilson \cite{GW} conjectured that the Gromov--Hausdorff limit of $(X_t,\ti{\omega}_t)$ is a Riemannian manifold of half dimension outside of a singular set of Hausdorff codimension at least $2$.
In the Strominger--Yau--Zaslow picture of mirror symmetry \cite{SYZ}, this limit should be the base of a special Lagrangian torus fibration on $X_t$ for $|t|$ small. This conjecture was proved for $K3$ surfaces in \cite{GW} under an assumption on the singular elliptic fibers, and in \cite{GTZ2} in general. In the case of large complex structure limits of certain hyperk\"ahler manifolds, this conjecture (except the statement about the singular set) was proved in \cite{GTZ}.

\section{Proofs of the main results}
\begin{proof}[Proof of Theorem \ref{main}]
$(a)\Rightarrow (b)$: $X_0$ is an irreducible Calabi--Yau variety with at worst canonical singularities, so if $f:\ti{X}_0\to X_0$ is a resolution of singularities, then we have
$$K_{\ti{X}_0}\sim \sum_E a_E E,$$
with $a_E\geq 0$. It follows that $K_{\ti{X}_0}$ is effective, as desired.

$(a)\Rightarrow (c)$ is Theorem B.1 in the appendix by M. Gross to Rong--Zhang \cite{RZ}.

$(c)\Leftrightarrow (d)$ is elementary: the fact that $\omega_t$ is Ricci-flat is equivalent to the complex Monge--Amp\`ere equation on $X_t, t\neq 0$,
\begin{equation}\label{vol1}
\omega_t^n=e^{c_t} (-1)^{\frac{n^2}{2}}\Omega_t\wedge\ov{\Omega_t},
\end{equation}
where $c_t\in\mathbb{R}$. Integrating this over $X_t$ we have
$$\int_{X_t}c_1(L_t)^n=\int_{X_t}\omega_t^n=e^{c_t} (-1)^{\frac{n^2}{2}}\int_{X_t}\Omega_t\wedge\ov{\Omega_t}.$$
Flatness of the family $\mathfrak{X}$ implies that $\int_{X_t}c_1(L_t)^n$ is constant in $t$. Hence $(c)$ is equivalent to $c_t\geq -C$ for all $t\neq 0$, which is in turn equivalent to $(d)$.

Let us also remark here that we always have that $c_t\leq C$ as $t$ approaches $0$. Indeed, a simple argument (see \cite[Theorem B.1 (i)]{RZ}) shows that
\begin{equation}\label{vol2}
(-1)^{\frac{n^2}{2}}\Omega_t\wedge\ov{\Omega_t}\geq C^{-1}\omega^n_{FS,t},
\end{equation}
where $\omega_{FS,t}$ is obtained by taking an embedding $\mathfrak{X}\subset \Delta\times\mathbb{CP}^N$ of the family with $m\mathcal{L}$ equal to the pullback of the hyperplane bundle, and
restricting $\frac{1}{m}$ times the Fubini--Study metric to $X_t$. In particular, we have that $\int_{X_t}\omega^n_{FS,t}=\int_{X_t}\omega_t^n$. Integrating \eqref{vol1} and \eqref{vol2} over $X_t$ then gives the desired bound $c_t\leq C$.

$(c)\Rightarrow (a)$: we learned this argument from S. Boucksom. Since $(\mathfrak{X},X_0)$ is dlt, we have that $X_0=\sum_{i}D_i$ is a reduced divisor with irreducible components $D_i$ which are normal (by \cite[Corollary 5.52]{KM}), and $X_0$ has simple normal crossings at the generic point of each component of $\cap_{j\in J}D_j$ for each $J\subset I$ (see \cite[4.16]{Ko}).
Assumption $(c)$ implies that
\begin{equation}\label{bound}
(-1)^{\frac{n^2}{2}}\int_{X_0^{reg}}\Omega_0\wedge\ov{\Omega_0}<\infty,
\end{equation}
where $\Omega_0=\Omega|_{X_0}$.
If $X_0$ is not irreducible, then there are two intersecting components $D_i\neq D_j$. By adjunction the restriction of $\Omega_0$ to $D_i$ has a pole of order at least $1$ along $D_i\cap D_j$, which is a smooth hypersurface of $D_i$ at the generic point of $D_i\cap D_j$. This contradicts \eqref{bound}.

Hence $X_0$ is irreducible, and as we remarked earlier this implies that $X_0$ is normal. Then \eqref{bound} is well-known to be equivalent to $X_0$ having at worst log-terminal singularities (see e.g. \cite[Lemma 6.4]{EGZ}). Therefore, given any resolution $f:\ti{X}_0\to X_0$ we have
$$K_{\ti{X}_0}\sim_{\mathbb{Q}} f^*K_{X_0}+\sum_E a_E E,$$
with $a_E>-1$. But since we assume that $K_{X_0}$ is Cartier, the coefficients $a_E$ are integers and therefore $a_E>-1$ is equivalent to $a_E\geq 0$, which shows that $X_0$ has at worst
canonical singularities.

$(b)\Rightarrow (a)$ is similar to the proof of \cite[Proposition 1.2]{Wa2}: we write again $X_0=\sum_{i}D_i$ as before, with $f:\ti{D}_1\to D_1$ a resolution of singularities with $H^0(\ti{D}_1,K_{\ti{D}_1})\neq 0$. First we show that $X_0=D_1$ is irreducible. If not,
since $X_0$ is connected, $D_1$ intersects $\sum_{j\neq 1}D_j.$ By adjunction we have that $K_{D_1}=-\sum_{j\neq 1}D_j|_{D_1}$, so $-K_{D_1}$ is effective and nontrivial. But then $K_{\ti{D}_1}\sim -f^*(\sum_{j\neq 1}D_j|_{D_1})+\sum_E a_E E$ (for some $a_E\in\mathbb{Z}$) cannot be effective because the divisors $E$ are $f$-exceptional, while $-f^*(\sum_{j\neq 1}D_j|_{D_1})$ is a pullback.

Hence $X_0=D_1$ is an irreducible normal variety with $K_{D_1}$ trivial, and $K_{\ti{D}_1}\sim \sum_E a_E E$ is effective, which implies that $a_E\geq 0$ and therefore $X_0$ has at worst canonical singularities.

$(c)\Rightarrow (e)$ follows from \cite[Theorem 2.1]{RZ}.

$(e)\Rightarrow (f)$: the Bishop--Gromov volume comparision theorem implies that for any $x\in X_t$ and any $0<r\leq \diam(X_t,\omega_t)$ we have
$$\frac{\vol_{\omega_t}B_{\omega_t}(x,r)}{c_nr^{2n}}\geq \frac{\int_{X_t}\omega_t^n}{c_n (\diam(X_t,\omega_t))^{2n}}\geq C^{-1},$$
using again that $\int_{X_t}\omega_t^n=\int_{X_t}c_1(L_t)^n$ is constant in $t$.

$(f)\Rightarrow (e)$ follows from a well-known elementary argument (see e.g. \cite[Lemma 3.3]{To}), using that the total volume of $(X_t,\omega_t)$ is constant. Indeed, if two points in $X_t$
are at very large $\omega_t$-distance, we can join them by a minimizing geodesic and construct many disjoint unit-size balls centered at points on the geodesic. By the volume non-collapsing hypothesis $(f)$, there cannot be too many such balls, since the total volume is bounded.
\end{proof}

\begin{proof}[Proof of Theorem \ref{cor}]
The condition that $0$ is at finite Weil--Petersson distance from $\Delta^\times$ is clearly unaffected by finite base changes. It is also unaffected by changing our family $\mathfrak{X}$ via a birational map which is an isomorphism over $\Delta^\times$. Indeed, let $\Phi:\mathfrak{X}'\dashrightarrow \mathfrak{X}$ be a birational map (over $\Delta$) between two such families, which induces isomorphisms $\Phi_t:X'_t\to X_t$ for $t\neq 0$. As before pick $\Omega_t, \Omega'_t$ trivializing sections of $K_{X_t}$ and $K_{X'_t}$, for $t\in\Delta^\times$. Then
$\Phi_t^*\Omega_t=a_t \Omega'_t$, with $a_t$ a never-vanishing holomorphic function on $\Delta^\times$. Therefore
$$(-1)^{\frac{n^2}{2}}\int_{X_t}\Omega_t\wedge\ov{\Omega_t}=|a_t|^2(-1)^{\frac{n^2}{2}}\int_{X'_t}\Omega'_t\wedge\ov{\Omega'_t},$$
and taking $-\ddbar\log$ of this equality shows that the two Weil--Petersson metrics we obtain on $\Delta^\times$ are equal.

Therefore one direction of this theorem follows from \cite[Theorem C]{Wa}. Conversely, we assume that $0$ is at finite Weil--Petersson distance from $\Delta^\times$.
Thanks to Mumford's semistable reduction theorem \cite{KKMS}, we may assume that $X_0$ is a snc divisor in $\mathfrak{X}$, in particular reduced.

We now use recent results in the Minimal Model Program, in the form given by \cite[Theorem 2.2.6]{NX} or \cite[Theorem 1.1]{Fu} (which rely on \cite{BCHM, HX, Lai}), to conclude that
our family $\mathfrak{X}\to\Delta$ is birational to a family $\pi:\mathfrak{X}'\to\Delta$, smooth over $\Delta^\times$,
with $\mathfrak{X}'$ normal, $K_{\mathfrak{X}'}$ $\mathbb{Q}$-Cartier, $K_{\mathfrak{X}'/\Delta}\sim_{\mathbb{Q}} 0$, which is also relatively minimal
(i.e. $(\mathfrak{X}',X'_0)$ is dlt). The birational map $\mathfrak{X}'\dashrightarrow \mathfrak{X}$ is an isomorphism over $\Delta^\times$, since the fibers $X_t,t\neq 0$ were already smooth Calabi--Yau manifolds and the birational map is obtained (through a minimal model program with scaling) by contracting relative extremal rays.

We now observe that in fact we have that $K_{\mathfrak{X}'/\Delta}\cong \mathcal{O}_{\mathfrak{X}'}$. This follows from the argument in \cite[Theorem 1.1]{Fu} (or \cite[Theorem 4.1.4]{NX}), as was explained to us by S. Boucksom. Indeed, we have that $K_{\mathfrak{X}'/\Delta}$ is $\mathbb{Q}$-Cartier and trivial over the generic fiber, so the sheaf $\mathcal{O}_{\mathfrak{X}'}(K_{\mathfrak{X}'/\Delta})$ is torsion-free. Therefore its pushforward $\pi_*\mathcal{O}_{\mathfrak{X}'}(K_{\mathfrak{X}'/\Delta})$
is also torsion-free, it has rank $1$ since the rank is computed at a generic point, and it lives on a curve hence it is a line bundle $L$. The natural injection $\pi^*L\to \mathcal{O}_{\mathfrak{X}'}(K_{\mathfrak{X}'/\Delta})$ is a surjection outside $X'_0$, hence $K_{\mathfrak{X}'/\Delta}=\pi^*L+D$, where $D$ is a Weil divisor supported on $X'_0$, which is also $\mathbb{Q}$-Cartier since it is a difference of two
$\mathbb{Q}$-Cartier divisors. Since $K_{\mathfrak{X}'/\Delta}$ is relatively nef, so is $D$, and the ``negativity lemma'' \cite[Lemma III.8.2]{BHPV} implies that $D=rX'_0$ for some $r\in\mathbb{Q}$. Since $X'_0$ is reduced, we have $D=\pi^*(r 0)$ which shows that $K_{\mathfrak{X}'/\Delta}$ is Cartier and (relatively) trivial.

Therefore, the family $\mathfrak{X}'\to \Delta$ satisfies the hypotheses of Theorem \ref{main}. We will show that Theorem \ref{main} $(b)$ holds, which implies that $X'_0$ has at worst canonical singularities, as required. Recall that without loss of generality we can assume that the family $\mathfrak{X}\to \Delta$ is semistable. Then \cite[Theorem 2.5]{Wa} shows that $0$ has finite Weil--Petersson distance from $\Delta^\times$ if and only if there is an irreducible component $D_1$ of $X_0$ with $H^0(D_1, K_{D_1})\neq 0$.
The birational map $\mathfrak{X}'\dashrightarrow \mathfrak{X}$ must induce a birational map from some component $D'_1$ of $X'_0$ to $D_1$, because otherwise the proper transform
of $D_1$ would be contracted at some step of the minimal model program with scaling, and hence $D_1$ would be uniruled \cite{Ka} which contradicts $H^0(D_1, K_{D_1})\neq 0$. Hence
$D'_1$ is birational to $D_1$, and so is any resolution $\ti{D}'_1$ of $D_1$. By the birational invariance of $h^{n,0}$, we conclude that $H^0(\ti{D}'_1, K_{\ti{D}'_1})\neq 0$,
and so the assumption in Theorem \ref{main} (b) holds.
\end{proof}

{\bf Acknowledgements. }I would like to thank S. Boucksom for pointing out a mistake in a previous version of this note, and for kindly allowing me to include some of his observations.
I would also like to thank T.C. Collins, H.-J. Hein, Y. Odaka, J. Song, G. Sz\'ekelyhidi and Y. Zhang for helpful discussions, and the referee for useful suggestions.


\begin{thebibliography}{99}
\bibitem{Al} Alexeev, V. {\em Higher-dimensional analogues of stable curves}, in {\em International Congress of Mathematicians. Vol. II}, 515--536, Eur. Math. Soc., Z\"urich, 2006.
\bibitem{BHPV} Barth, W.P., Hulek, K., Peters, C.A.M., Van de Ven, A. \emph{Compact complex surfaces}, Springer-Verlag, 2004.
\bibitem{BCHM} Birkar, C., Cascini, P., Hacon, C.D., McKernan, J. {\em Existence of minimal models for varieties of log general type}, J. Amer. Math. Soc. {\bf 23} (2010), no. 2, 405--468.
\bibitem{DS} Donaldson, S.K., Sun, S. {\em Gromov-Hausdorff limits of K\"ahler manifolds and algebraic geometry}, Acta Math. {\bf 213} (2014), no. 1, 63--106.
\bibitem{EGZ} Eyssidieux, P., Guedj, V., Zeriahi, A. \emph{Singular K\"ahler-Einstein metrics}, J. Amer. Math. Soc. {\bf 22}  (2009), 607--639.
\bibitem{Fu} Fujino, O. {\em Semi-stable minimal model program for varieties with trivial canonical divisor}, Proc. Japan Acad. Ser. A Math. Sci. {\bf 87} (2011), no. 3, 25--30.
\bibitem{GTZ} Gross, M., Tosatti, V., Zhang, Y. {\em Collapsing of abelian fibred Calabi-Yau manifolds}, Duke Math. J. {\bf 162} (2013), no. 3, 517--551.
\bibitem{GTZ2} Gross, M., Tosatti, V., Zhang, Y. {\em Gromov-Hausdorff collapsing of Calabi-Yau manifolds}, preprint, arXiv:1304.1820.
\bibitem{GW} Gross, M., Wilson, P.M.H. \emph{Large complex structure limits of $K3$ surfaces}, J. Differential Geom. {\bf 55} (2000), no. 3, 475--546.
\bibitem{HX} Hacon, C.D., Xu, C. {\em Existence of log canonical closures}, Invent. Math. {\bf 192} (2013), no. 1, 161--195.
\bibitem{Ka} Kawamata, Y. {\em On the length of an extremal rational curve}, Invent. Math. {\bf 105} (1991), no. 3, 609--611.
\bibitem{KKMS} Kempf, G., Knudsen, F.F., Mumford, D., Saint-Donat, B. {\em Toroidal embeddings. I}, Lecture Notes in Mathematics, Vol. 339, Springer-Verlag, 1973.
\bibitem{Ko} Koll\'ar, J. {\em Singularities of the minimal model program}, Cambridge University Press, 2013.
\bibitem{KM} Koll\'ar, J., Mori, S. {\em Birational geometry of algebraic varieties}, Cambridge University Press, 1998.
\bibitem{KS} Kontsevich, M., Soibelman, Y. {\em  Homological mirror symmetry and torus fibrations}, in \it Symplectic geometry and mirror symmetry, \rm 203--263, World Sci. Publishing 2001.
\bibitem{Lai} Lai, C.-J. {\em Varieties fibered by good minimal models}, Math. Ann. {\bf 350} (2011), no. 3, 533--547.
\bibitem{Mo}  Morrison, D.R. {\em Compactifications of moduli spaces inspired by mirror symmetry}, in
 {\em Journ\'ees de G\'eom\'etrie Alg\'ebrique d'Orsay (Orsay, 1992)}, Ast\'erisque No. {\bf 218} (1993), 243--271.
\bibitem{NX} Nicaise, J., Xu, C. {\em The essential skeleton of a degeneration of algebraic varieties}, preprint, arXiv:1307.4041.
\bibitem{O1} Odaka, Y. {\em The Calabi conjecture and K-stability}, Int. Math. Res. Not. IMRN {\bf 2012}, no. 10, 2272--2288.
\bibitem{O2} Odaka, Y. {\em The GIT stability of polarized varieties via discrepancy}, Ann. of Math. (2) {\bf 177} (2013), no. 2, 645--661.
\bibitem{RZ} Rong, X., Zhang, Y. {\em Continuity of extremal transitions and flops for Calabi-Yau manifolds}, J.  Differential Geom. {\bf 89} (2011), no. 2, 233--269.
\bibitem{RZ2} Rong, X., Zhang, Y. {\em  Degenerations of Ricci-flat Calabi-Yau manifolds}, Commun. Contemp. Math. {\bf 15} (2013), no. 4, 1250057, 8 pp.
\bibitem{Sc} Schmid, W. {\em Variation of Hodge structure: the singularities of the period mapping}, Invent. Math. {\bf 22} (1973), 211--319.
\bibitem{SYZ} Strominger, A., Yau, S.-T., Zaslow, E. {\em Mirror symmetry is $T$-duality},  Nuclear Phys. B  {\bf 479}  (1996),  no. 1-2, 243--259.
\bibitem{Ta1} Takayama, S. {\em On uniruled degenerations of algebraic varieties with trivial canonical divisor}, Math. Z. {\bf 259} (2008), no. 3, 487--501.
\bibitem{Ta2} Takayama, S. {\em On moderate degenerations of polarized Ricci-flat K\"ahler manifolds}, preprint, 2014.
\bibitem{To} Tosatti, V. {\em Adiabatic limits of Ricci-flat K\"ahler metrics}, J. Differential Geom. {\bf 84} (2010), no.2, 427--453.
\bibitem{Wa} Wang, C.-L. {\em On the incompleteness of the Weil-Petersson metric along degenerations of Calabi-Yau manifolds}, Math. Res. Lett. {\bf 4} (1997), no. 1, 157--171.
\bibitem{Wa2} Wang, C.-L. {\em Quasi-Hodge metrics and canonical singularities}, Math. Res. Lett. {\bf 10} (2003), no. 1, 57--70.
\bibitem{Wi} Wilson, P.M.H. \emph{Metric limits of Calabi-Yau manifolds}, in \emph{The Fano Conference}, 793--804, Univ. Torino, Turin, 2004.
\end{thebibliography}
 \end{document}